\documentclass[12pt]{amsart}

\usepackage{amsmath}
\usepackage{amssymb}  
\usepackage{latexsym} 
\usepackage{comment}
\usepackage{url}
\usepackage{fullpage,url,amssymb,amsmath,amsthm,amsfonts,mathrsfs}
\usepackage[usenames,dvipsnames]{color}
\usepackage[pagebackref = true, colorlinks = true, linkcolor = blue, citecolor = Green]{hyperref}
\usepackage[alphabetic,lite]{amsrefs}
\usepackage{enumitem}
\usepackage{amscd}   
\usepackage[all, cmtip]{xy} 
\usepackage{xfrac}

\DeclareFontEncoding{OT2}{}{} 
\newcommand{\textcyr}[1]{%
 {\fontencoding{OT2}\fontfamily{wncyr}\fontseries{m}\fontshape{n}\selectfont #1}}

\usepackage[all]{xy}
\usepackage{fullpage}
\newcommand{\Sha}{{\mbox{\textcyr{Sh}}}}

\usepackage{color} 

\def\act#1#2%
  {\mathop{}%
   \mathopen{\vphantom{#2}}^{#1}%
   \kern-3\scriptspace%
   #2}

\newcommand{\Z}{{\mathbb Z}}
\newcommand{\Q}{{\mathbb Q}}

\newcommand{\F}{{\mathbb F}}

\newcommand{\Kbar}{{\overline{K}}}
\newcommand{\Cbar}{{\overline{C}}}

\DeclareMathOperator{\disc}{disc}

\DeclareMathOperator{\divv}{div}

\DeclareMathOperator{\Div}{Div}

\DeclareMathOperator{\Pic}{Pic}

\DeclareMathOperator{\Spec}{Spec}

\DeclareMathOperator{\Sel}{Sel}
\DeclareMathOperator{\Cov}{Cov}
\DeclareMathOperator{\SL}

\newtheorem{Theorem}{Theorem}[section]
\newtheorem{Lemma}[Theorem]{Lemma}

\newtheorem{Remark}[Theorem]{Remark}

\numberwithin{equation}{section}

\begin{document}

\title{Improved rank bounds from $2$-descent on hyperelliptic Jacobians}

\author{Brendan Creutz}
\address{School of Mathematics and Statistics, University of Canterbury, Private Bag 4800, Christchurch 8140, New Zealand}
\email{brendan.creutz@canterbury.ac.nz}
\urladdr{http://www.math.canterbury.ac.nz/\~{}bcreutz}

\maketitle
\begin{abstract}
	We describe a qualitative improvement to the algorithms for performing $2$-descents to obtain information regarding the Mordell-Weil rank of a hyperelliptic Jacobian. The improvement has been implemented in the Magma Computational Algebra System and as a result, the rank bounds for hyperelliptic Jacobians are now sharper and have the conjectured parity.
\end{abstract}

\section{Introduction}

	Suppose $X$ is a smooth projective and geometrically irreducible curve over a global field $k$. It is an open question whether or not there is an algorithm to compute the set $X(k)$ of rational points on $X$. A related question is the determination of the group $J(k)$ of rational points on the Jacobian $J$ of $X$. In the case that $X$ is a hyperelliptic curve and $k$ has characteristic different from $2$, the method of $2$-descent as described in \cite{BruinStoll,Cassels,PoonenSchaefer} is sometimes successful in practice. In \cite{CreutzANTSX} it is shown how to incorporate additional information coming from a $2$-descent on the variety $J^1 = \Pic^1_X$ which is a torsor under $J$. In particular an algorithm for computing a set denoted $\Sel^2_{\textup{alg}}(J^1/k)$ is given and the following result is proven.
	
	\begin{Theorem}[{\cite[Theorem 4.5 and Corollary 4.6]{CreutzANTSX}}]\label{thm:v1}
		Let $X$ be a hyperelliptic curve over a global field of characteristic different from $2$. Suppose that $X$ is everywhere locally solvable. Then $\Sel_{\textup{alg}}^{2}(J^1/k)$ is nonempty if and only if the torsor $J^1$ is divisble by $2$ in $\Sha(J/k)$. Moreover, if $\Sel^2_{\textup{alg}}(J^1/k) = \emptyset$, then $\dim_{\F_2}\Sha(J/k)[2] \ge 2$.
	\end{Theorem}
	
	The second statement of the theorem is deduced from the first using the fact that the group $\Sha(J/k)[2]/2\Sha(J/k)[4]$ has square order, a consequence of the fact that the Cassels-Tate pairing induces a nondegenerate alternating pairing on this quotient because $X$ is assumed to have divisors of degree $1$ everywhere locally \cite{PoonenStoll}. This is useful in determining the Mordell-Weil rank of the Jacobian since there is an exact sequence,
	\[
		0 \to J(k)/2J(k) \to \Sel^2(J/k) \to \Sha(J/k)[2] \to 0\,.
	\]
	Thus lower bounds for $\Sha(J/k)[2]$ allow one to deduce sharper upper bounds for the rank of $J(k)$. As remarked on \cite[p. 305]{CreutzANTSX}, the hypothesis of Theorem~\ref{thm:v1} that $X$ be everywhere locally solvable seems overly strict; one would expect that the theorem remains true under the weaker hypothesis that $X$ has a rational divisor of degree $1$ everywhere locally. The purpose of this short note is to show that this is indeed the case. The key new ingredient is part of recent work of Bhargava-Gross-Wang \cite{BGW} concerning the $2$-Selmer set of $J^1$. Using this we prove the following result.

	\begin{Theorem}\label{thm:v2}
		Let $X$ be a hyperelliptic curve over a global field of characteristic different from $2$. Suppose $\Div^1(X_{k_v}) \ne \emptyset$ for all completions $v$ of $k$. Then $\Sel_{\textup{alg}}^{2}(J^1/k)$ is nonempty if and only if the torsor $J^1$ is divisible by $2$ in $\Sha(J/k)$. Moreover, if $\Sel^2_{\textup{alg}}(J^1/k) = \emptyset$, then $\dim_{\F_2}\Sha(J/k)[2] \ge 2$.
	\end{Theorem}

	This improvement was motivated in part by a question of Michael Stoll, who noted that the rank bounds for Mordell-Weil groups of hyperelliptic Jacobians over $\Q$ computed by Magma \cite{Magma} did not always have the parity one would expect assuming standard conjectures. Unlike the usual $2$-descent on the Jacobian,  computing $\Sel_\textup{alg}^2(J^1/k)$ and using Theorem~\ref{thm:v1} can indeed lead to bounds which can be improved by assuming parity or finiteness of $\Sha(J/k)$. Specifically, if $X$ has divisors of degree $1$ everywhere locally and $\Sel^2_\textup{alg}(J^1/k) = \emptyset$, then $J^1(k) = \emptyset$ and $J^1$ represents a nontrivial element of $\Sha(J/k)[2]$. However, if $X$ does not have points everywhere locally Theorem~\ref{thm:v1} does not apply and, without further assumptions, we cannot conclude that $\Sha(J/k)[2]/2\Sha(J/k)[4]$ is nontrivial. Consequently we only get a lower bound of $1$ for $\dim_{\F_2}\Sha(J/k)[2]$ instead of $2$. If $\Sha(J/k)$ contains no non-trivial infinitely $2$-divisible elements (as is widely conjectured but far from being proven), then $\dim_{\F_2}\Sha(J/k)[2]$ is even \cite{PoonenStoll}. So in such cases the bound obtained does not have the expected parity, but it is the best unconditional result that one can deduce using Theorem~\ref{thm:v1}.

	In the situation described above Theorem~\ref{thm:v2} applies, allowing us to repair this defect. This improvement has been implemented in Magma and, as a result, the rank bounds for hyperelliptic Jacobians are now sharper and have the expected parity.

A particular example is given by the curve
	\[
		X/\Q : y^2 = 5x^6 + x^5 + x^4 - 4x^3 - 4x^2 + 5x - 1\,.
	\]
	The polynomial on the right hand side above has Galois group $S_6$ over $\Q$, so $J(\Q)$ has no points of order $2$. The $2$-Selmer group of the Jacobian is isomorphic to $\Z/2\times \Z/2$, from which we deduce that the rank of $J(\Q)$ is at most $2$. Computing $\Sel^2_{\textup{alg}}(J^1/\Q)$ as described in \cite{CreutzANTSX} we find that it is empty. Since $\Div^1(X_{\Q_p}) \ne \emptyset$ for all primes $p \le \infty$, this implies that $J^1(\Q) = \emptyset$ and, hence, that $J^1$ represents a nontrivial element of $\Sha(J/\Q)[2]$. However, $X(\Q_3) = \emptyset$, so Theorem~\ref{thm:v1} above does not apply. So, without Theorem~\ref{thm:v2}, the best unconditional bound for the rank of $J(\Q)$ we can get is $1$. Whereas Theorem~\ref{thm:v2} gives that $\Sha(J/\Q)[2]$ has rank at least $2$ and therefore that $J(\Q)$ has rank $0$, unconditionally. 

	\section{The proof of Theorem~\ref{thm:v2}}
		
		Let $X/k$ be the hyperelliptic curve given by the affine equation $y^2 = f(x)$ with $f(x) \in k[x]$ a square free polynomial of even degree $n$. For any field extension $K$ of $k$, let $\Cov^2(J^1/K)$ denote the set of $K$-isomorphism classes of $2$-coverings of $J_K = J \times_{\Spec(k)}\Spec(K)$. Given a pair of symmetric bilinear forms $(A,B)$ such that $\disc(Ax-B) = f(x)$, the Fano variety of maximal linear subspaces contained in the base locus of the pencil of quadrics generated by $(A,B)$ may be given the structure of a $2$-covering of $J^1$ \cite[Theorem 23]{BGW}. In general, not every $2$-covering of $J^1$ can be defined in this way.  Let $\Cov_0^2(J^1/K)$ denote the subset of $\Cov^2(J^1/K)$ consisting of those isomorphism classes of $2$-coverings that do arise in this way from a pair of symmetric bilinear forms. This set is in bijection with the set of $(\operatorname{SL}_n/\mu_2)(K)$ orbits of pairs $(A,B)$ with $\disc(Ax-B) = f(x)$. We will show below that $\Cov_0^2(J^1/K)$ is equal to the set $\Cov^2_\textup{good}(J^1/K)$ defined in \cite{CreutzANTSX}. But before that, let us outline the proof of Theorem~\ref{thm:v1} given in \cite{CreutzANTSX}, and in so doing see how this allows us to deduce Theorem~\ref{thm:v2}.

		In \cite[Section 6]{CreutzANTSX} we defined a `descent map' on $\Cov_\textup{good}^2(J^1/k)$ and proved that it induces a bijection $\Cov_\textup{good}^2(J^1/k) \cap \Sel^2(J^1/k) \to \Sel_{\textup{alg}}^2(J^1/k)$. This does not require local solubility of $X$; it only requires the weaker assumption that $X$ has divisors of degree $1$ everywhere locally. The second step is to prove that the set $\Sel^2(J^1/k)$ of locally soluble coverings is contained in $\Cov_\textup{good}^2(J^1/k)$. This was shown in \cite[Proposition 6.2]{CreutzANTSX} under the assumption that $X$ is everywhere locally solvable. When combined with the descent map, this implies that $\Sel^2(J^1/k)$ and $\Sel_{\textup{alg}}^2(J^1/k)$ are in bijection, and the conclusion of the theorems follows. This same argument proves Theorem~\ref{thm:v2}, provided we can verify that $\Sel^2(J^1/k) \subset \Cov_{\textup{good}}^2(J^1/k)$ under the weaker hypothesis that $X$ has divisors of degree $1$ everywhere locally.

		Bhargava, Gross and Wang show that $\Sel^2(J^1/k) \subset \Cov_0^2(J^1/k)$ when $X$ has divisors of degree $1$ everywhere locally \cite[Theorem 31]{BGW}. It is therefore enough to show that $\Cov_0^2(J^1/k) \subset \Cov^2_{\textup{good}}(J/k)$. We show that these sets are actually equal.

		\begin{Lemma}
			Let $K$ be any field extension of $k$. Then $\Cov_0^2(J^1/K) = \Cov^2_{\textup{good}}(J/K)$.
		\end{Lemma}

		\begin{proof}
		By geometric class field theory, pulling back along the canonical map $X \to \Pic^1_X = J^1$ gives a bijective map $\Cov^2(J^1/K) \to \Cov^2(X/K)$, where $\Cov^2(X/K)$ is the set of isomorphism classes of $2$-coverings of $X$, i.e., $K$-forms of the maximal unramified exponent $2$ abelian covering of $X$. Since this map is injective, it is enough to show that $\Cov_0^2(J^1/K)$ and $\Cov^2_\textup{good}(J^1/K)$ have the same image in $\Cov^2(X/K)$.

		The image $\Cov_\textup{good}^2(X/K)$ of $\Cov_\textup{good}^2(J^1/K)$ is by definition (see \cite[Definition 5.3]{CreutzANTSX}) the set of isomorphism classes $\pi: C \to X_K$ with the property that $\pi^*\omega$ is linearly equivalent to a $K$-rational divisor for any Weierstrass point $\omega \in X(\Kbar)$. By \cite[Theorem 24]{BGW} the image $\Cov_0^2(X/K)$ of $\Cov_0^2(J^1/K)$ is the set of isomorphism classes $C \to X_K$ which admit a lift to a degree $2$ covering $C' \to C$ such that the composition $C' \to X_K$ is a $K$-form of the maximal abelian of exponent $2$ covering of $X_{\Kbar}$ unramified outside $\frak{m}$, for $\frak{m} \in \Div(X)$ a (fixed) divisor of degree $2$ corresponding to a pair of non-Weierstrass points conjugate under the hyperelliptic involution.

		First we show that $\Cov_\textup{good}^2(X/K) \subset \Cov_0^2(X/K)$. Suppose $\pi: C \to X_K$ represents a class in $\Cov_\textup{good}^2(X/K)$. Then there is some $d \in \Div(C)$ linearly equivalent to $\pi^*\omega$. There is $g \in \Kbar(X)^\times$ such that $2\omega - \frak{m} = \divv(g)$. Then $\divv(\pi^*g) = 2\pi^*\omega - \pi^*\frak{m}$ is linearly equivalent to the $K$-rational principal divisor $2d - \pi^*\frak{m} = \divv(f)$. By Hilbert's Theorem 90 we may assume $f \in K(C)^\times$. The degree $2$-cover corresponding to the quadratic extension $K(C)(\sqrt{f})$ gives the desired lift $C' \to C$.

		Now let us show that $\Cov_0^2(X/K) \subset \Cov_\textup{good}^2(X/K)$. Suppose $\pi : C \to X_K$ represents a class in $\Cov_0^2(X/K)$. For each Weierstrass point $\omega \in X(\Kbar)$ there is a function $f_\omega \in \Kbar(X)$ such that $\divv(f_\omega) = 2\omega - \frak{m}$. The maximal exponent $2$ abelian covering of $X_{\Kbar}$ unramified outside $\frak{m}$ corresponds to the extension obtained by adjoining square roots of all $f_\omega$, while its maximal unramified subcover corresponds to the extension obtained by adjoing square roots to all ratios $f_\omega/f_{\omega'}$. From this one sees that $C_\Kbar' \to C_\Kbar$ is the double cover ramified at $\pi^{-1}(\frak{m}) = \pi^*\frak{m} \subset \Div(C_\Kbar)$. Any $K$-form $C' \to C$ of this must be given by adjoining a square root of a function $f \in K(C)^\times$ with divisor of the form $\divv(f) = 2d - \pi^*\frak{m}$, for some $d \in \Div(C)$. Then for any Weierstrass point $\omega$, we have that $2d - 2\pi^*\omega \in \Div(\Cbar)$ is principal. But since $C \to X_K$ is a {\it maximal} unramified exponent $2$ abelian covering of $X_K$ we must have that $d - \pi^*\omega$ is principal, i.e., that $\pi^*\omega$ is linearly equivalent to a $k$-rational divisor. Hence $C \to X_K$ represents a class in $\Cov_{\textup{good}}(X/K)$.
		\end{proof}

	\begin{Remark}
		If $X$ does not have divisors of degree $1$ everywhere locally, then \cite{BGW} shows that $\Sel^2(J^1/k) \cap \Cov_0^2(J^1/k) = \Sel^2(J^1/k) \cap \Cov_{\textup{good}}^2(J^1/k) = \emptyset$. In this situation one also has $\Sel_{\textup{alg}}^2(J^1/k) = \emptyset$  (cf. \cite[Remark, p. 305]{CreutzANTSX}). It is still possible in this situation, however, that $\Sel^2(J^1/k) \ne \emptyset$. This shows that it is not possible to generalize the theorem further to the case that $X$ does not have divisors of degree $1$ everywhere locally, even if $J^1$ is assumed to have points everywhere locally.
	\end{Remark}
%

	\subsection*{Acknowledgements}
		The author would like to thank Steve Donnelly and Michael Stoll for helpful discussions.


\end{document}